\documentclass[11pt,a4paper]{amsart}
\usepackage[top=1.33in, bottom=1.33in, left=1.25in, right=1.25in]{geometry}
\usepackage{amsmath, amsxtra, amsthm, amsfonts, amssymb, mathtools}
\usepackage{mathrsfs}
\usepackage{graphicx}
\usepackage{url}
\usepackage[dvipsnames]{xcolor}
\usepackage{bbm}
\usepackage{bm}
\usepackage{tikz-cd}
\usepackage{tikz}
\usetikzlibrary{backgrounds}
\usepackage[cal=boondoxo,scr=euler]{mathalfa}
\usepackage{float}
\usetikzlibrary{decorations.pathreplacing, positioning}
\usetikzlibrary{decorations.markings}
\usetikzlibrary{shapes,positioning,intersections,quotes}
\usetikzlibrary{shapes,fit}
\usetikzlibrary{arrows.meta, calc}
\definecolor{curvyarrow}{RGB}{200, 100, 150}
\usepackage{comment}
\usepackage[shortlabels]{enumitem} 
\usepackage{enumitem}
\usepackage{relsize}
\usepackage{stmaryrd}

\usepackage[all,cmtip]{xy}	
\xyoption{arrow}
\usepackage{mlmodern} 
\usepackage[T1]{fontenc}
\usepackage{enumitem}

\usepackage[backref=page,linktocpage]{hyperref} 
\hypersetup{
    colorlinks,
    allcolors=teal
}

\usepackage[capitalize]{cleveref}

\newtheoremstyle{results}
{8pt}
{8pt}
{\itshape}
{}
{\bfseries}
{}
{.5em}
{}
\theoremstyle{results}
\newtheorem{thm}{Theorem}[section]

\newtheorem{corollary}[thm]{Corollary}
\newtheorem{lemma}[thm]{Lemma}
\newtheorem{maintheorem}[thm]{Theorem}
\newtheorem{claim}[thm]{Claim}
\newtheoremstyle{definitions}
{7pt}
{7pt}
{}
{}
{\bfseries}
{}
{.5em}
{}

\theoremstyle{definitions}
\newtheorem{definition}[thm]{Definition}
\newtheorem{conjecture}[thm]{Conjecture}

\newtheorem{question}[thm]{Question}

\crefname{thm}{theorem}{theorems}
\Crefname{thm}{Theorem}{Theorems}
\crefname{lemma}{lemma}{lemmas}
\Crefname{lemma}{Lemma}{Lemmas}
\crefname{proposition}{proposition}{propositions}
\Crefname{proposition}{Proposition}{Propositions}
\crefname{corollary}{corollary}{corollaries}
\Crefname{corollary}{Corollary}{Corollaries}
\crefname{definition}{definition}{definitions}
\Crefname{definition}{Definition}{Definitions}
\crefname{conjecture}{conjecture}{conjectures}
\Crefname{conjecture}{Conjecture}{Conjectures}
\crefname{example}{example}{examples}
\Crefname{example}{Example}{Examples}
\crefname{remark}{remark}{remarks}
\Crefname{remark}{Remark}{Remarks}

\def\Z{\mathbb{Z}}

\def\P{\mathbb{P}}

\def\Z{\mathbb{Z}}

\newcommand{\abs}[1]{\left\lvert#1\right\rvert}

\newcommand{\rk}{\operatorname{rk}}

\renewcommand{\L}{\mathcal{L}}
\newcommand{\zero}{\hat{0}}

\newcommand{\des}{\mathrm{des}}

\newcommand{\un}{\hat{1}}
\newcommand{\At}{\mathrm{At}}

\newcommand{\Edge}{\mathcal{E}}

\newcommand{\medbullet}{\mathbin{\vcenter{\hbox{\scalebox{0.6}{$\bullet$}}}}}

\renewcommand{\L}{\mathcal{L}}

\newcommand{\Des}{\mathrm{Des}}

\renewcommand{\P}{\mathrm{P}}

\renewcommand{\hat}{\widehat}

\newcommand{\ind}{\mathrm{Ind}}

\title{An EL-shellable rank-uniform poset that is not UMEL-shellable}

\author{Hadi Al Aithan}
\address{H. Al Aithan, Massachusetts Institute of Technology, Cambridge, USA}
\email{hadi19@mit.edu}

\date{\today}

\begin{document}
\begin{abstract}
     We address Questions~8.1 and~8.4 of Coron--Ferroni--Li \cite{coron2025chow}. For the first part of Question~8.1, which asks whether every rank-uniform EL-shellable poset admits a rank-uniform EL-labeling, we construct a counterexample of rank $4$ and prove that this rank is minimal. For the second part of Question~8.1, concerning whether every rank-uniform EL-shellable poset is UMEL-shellable, we construct a counterexample of rank $3$. We also develop a general criterion for constructing UMEL-shellable posets of rank $3$. As an application, we give an affirmative answer to Question~8.4 of Coron--Ferroni--Li.
\end{abstract}

\maketitle
\tableofcontents

\section{Introduction}

    The study of polytopes and matroids and related polynomials is a rich and active field of combinatorics. For example, given a polytope or a simplicial complex, one can study its $h$-polynomial and $\gamma$-polynomial, see \cite{postnikov2008faces,  brenti2008f, athanasiadis2012flag, aisbett2020geometric, nguyen2024identity, nguyen2025p, nguyen2027poset}. Similarly, given a matroid, one can study its Chow polynomial, see \cite{feichtner2004chow, hameister2018chow, braden2022semi, backman2023simplicial, ferroni2024hilbert, liao2024equivariant, branden2025chow, stump2025chow, hoster2025chow}.

    Among the adjectives for polynomials with nonnegative coefficients, real-rooted is the strongest as it implies log-concavity and unimodality. It is also the hardest, and the following conjecture remains open not for lack of effort.

    \begin{conjecture}[\cite{stevens2021real, ferroni2024valuative}]
        The Chow polynomial of a matroid has only nonpositive real roots.
    \end{conjecture}

    In \cite{coron2025chow}, Coron--Ferroni--Li introduced UMEL-shellable posets and proved that their Chow polynomials and augmented Chow polynomials are real-rooted. Loosely speaking, UMEL-shellable posets are \textit{rank-uniform} EL-shellable posets in which the EL labeling is \textit{monotonic}. Clearly, UMEL-shellable posets are EL-shellable. However, the converse is not clear. Thus, Coron--Ferroni--Li asked the following questions.

    \begin{question}[{\cite[Question 8.1]{coron2025chow}}]
        If $P$ is rank-uniform and EL-shellable, does $P$ always admit an EL-labeling \(\lambda\) such that \((P, \lambda)\) is a monotonically rank-uniform labeled poset?
    \end{question}
    
    \begin{question}[{\cite[Question 8.1]{coron2025chow}}]
        If $P$ is rank-uniform and EL-shellable, does $P$ always admit an EL-labeling \(\lambda\) such that \((P, \lambda)\) is a rank-uniform labeled poset?
    \end{question}
    
    \begin{question}[{\cite[Question 8.4]{coron2025chow}}]
        Let $\L$ be a geometric lattice of rank $3$. If $\L$ is rank-uniform, does there exist a total order $\vartriangleleft$ on $\At(\L)$ such that the quantity 
        \begin{equation*}
            \des_0(s) = |\{ t \gtrdot s \, |\, \exists s' \vartriangleleft s \textrm{ s.t. } s' \leqslant t \}|
        \end{equation*}
        defined for all atoms $s$, is weakly increasing with respect to $\vartriangleleft$?  
    \end{question}

    Questions~1.2 and~1.3 are immediate for posets of rank $2$. In Section~\ref{rank-4}, we provide counterexamples showing that the corresponding statements do not hold in general in higher rank. However, we give affirmative answers to both questions for rank-$3$ posets in Theorems~\ref{theorem-1} and~\ref{theorem-2}. These results, together with our atom-insertion construction, yield Theorem~\ref{theorem-3}, which gives an affirmative answer to Question~1.4 and, in fact, applies to a broader class of rank-$3$ posets than geometric lattices.

\subsection*{Acknowledgments} We thank Son Nguyen for his mentorship throughout the development of this work. His guidance during the summer led to the study of this topic and helped identify the most fruitful directions to pursue. We are also grateful to Ali Alsaleh for his assistance in developing and running, using Claude Code, the computer program used to search for counterexamples to Question~8.1 of Coron--Ferroni--Li~\cite{coron2025chow}. Discussions with Ali also contributed to a deeper understanding of the underlying concepts and helped improve the results presented in this paper.

\section{Definitions}
\label{sec:definitions}
In this section, we recall the definitions from \cite{coron2025chow} that will be needed throughout this paper. For a more detailed discussion of rank-uniform labeled posets, we refer the reader to \cite[Section 3]{coron2025chow}.

\subsection{Posets and edge labelings}
Let $(\P, \leqslant)$ be a finite graded bounded poset with rank function $\rk \colon \P \to \mathbb{Z}_{\geqslant 0}$. We denote by $\un$ the upper bound of $\P$ and by $\zero$ its lower bound. We denote by $\At(\P)$ the set of elements of rank $1$, called the atoms of $\P$. For all $s\leqslant t \in \P$, the interval between $s$ and $t$ is denoted by 
\[
[s, t] \coloneqq \{ u \in \P \mid s \leqslant u \leqslant t\}. 
\]
If $s < t \in \P$ are such that no $u \in \P$ satisfies $s < u < t$, we say that $s$ is covered by $t$, or $t$ covers $s$, and we write $s \lessdot t$. We denote by $\Edge(\P)$ the set of pairs $(s, t) \in \P \times \P$ such that $s \lessdot t.$ An edge labeling of $\P$ is a map $\lambda: \Edge(\P) \rightarrow \Lambda$, with $\Lambda$ a totally ordered set with total order denoted $\vartriangleleft$. An edge labeling of $\P$ allows us to associate to every saturated chain $C$ of the form 
\[
s = u_0 \lessdot u_1 \lessdot \cdots \lessdot u_k = t 
\] between two elements $s \leqslant t$ the list of labels $\lambda(C) \coloneqq \lambda(u_0, u_1)\ldots \lambda(u_{k-1}, u_k).$

\begin{definition}
An edge labeling $\lambda: \Edge(\P) \rightarrow \Lambda$ of $\P$ is called \textit{edge-lexicographical} (EL), if for all $s \leqslant t$ in $\P$, there exists a unique maximal chain $C$
\[
s = u_0 \lessdot u_1 \lessdot \cdots \lessdot u_{k} = t
\]
between $s$ and $t$ with weakly increasing labels as the rank increases: 
\[
\lambda(u_0, u_1) \trianglelefteqslant \lambda(u_1, u_2) \trianglelefteqslant \cdots \trianglelefteqslant \lambda(u_{k-2}, u_{k-1}) \trianglelefteqslant \lambda(u_{k-1}, u_k),
\] and the maximal chain $C$ is lexicographically minimal among all maximal chains from $s$ to $t$. 
\end{definition}

An \emph{EL-labeled poset} $(\P,\lambda)$ is a pair consisting of a poset $\P$ and an EL-labeling 
\[
\lambda:\Edge(\P)\rightarrow \Lambda.
\]

\subsection{Rank-uniform labeled posets}
Let $\P$ be a bounded finite graded poset with rank function $\rk \colon \P \to \mathbb{Z}_{\geqslant 0}$. For every element $s$ and for every nonnegative integer $k$, we write $W_{k}(s)$ for the number of elements of rank $k + \rk s$ above $s$, called the \emph{$k$-th Whitney number above $s$}. 
\begin{definition}
The poset $\P$ is \textit{rank-uniform} if for all $s, t \in \P$ of the same rank and for every positive integer $k$, the number of elements of rank $k + \rk s$ above $s$ is equal to the number of elements of rank $k + \rk s$ above $t$: 
\begin{equation*}
    W_k(s) = W_{k}(t).
\end{equation*}
\end{definition}

Coron--Ferroni--Li introduced this terminology to define an edge-labeled version of the above notion for a poset equipped with an EL-labeling. 

\begin{definition}[Width, index, descent]Let $(\P, \lambda)$ be an EL-labeled poset and let $s \in \P$. Denote by $\lambda_1 \vartriangleleft \cdots \vartriangleleft \lambda_{\ell}$ the labels of the covering relations above $s$.
\begin{itemize}[leftmargin=23pt,itemsep=-1pt]
    \item The \emph{width above $s$} is the function 
    \[
    \omega_s \colon \mathbb{Z}_{>0} \to \Z_{\geqslant 0}, \quad \omega_s(i) \coloneqq \abs{ \: \{t \in \P \mid s \lessdot t\: , \: \lambda(s, t) = \lambda_i \: \} \: }. 
    \] 
    In other words, the width above $s$ at index $i$ records the number of covers of $s$ labeled by $\lambda_i$ (see Figure \ref{fig:width} below). By convention, for $i > \ell$, $\omega_{s}(i) = 0$.
     \begin{figure}[H]
         \centering
         \begin{tikzpicture}[scale=0.60, baseline=14px]
         \node[circle, draw=black, fill=red!15, inner sep=1.7pt] (A) at (0,0) {$s$};
         \node[circle, draw=black, fill=red!15, inner sep=1.5pt] (B) at (-6,2) {};
        \node[circle, draw=black, fill=red!15, inner sep=1.5pt] (C) at (-2,2) {};
         \node[circle, draw=black, fill=red!15, inner sep=1.5pt] (D) at (0,2) {};
         \node[circle, draw=black, fill=red!15, inner sep=1.5pt] (E) at (4,2) {};
        
         \draw[thick, BlueViolet] (A) -- (B) node[pos=0.5, left=6pt, scale = 0.8] {$\lambda_1$};
        \draw[thick, BlueViolet] (A) -- (C) node[midway,left, scale = 0.8] {$\lambda_1$};
         \draw[thick, BlueViolet] (A) -- (D) node[midway,left, scale = 0.8] {$\lambda_2$};
         \draw[thick, BlueViolet] (A) -- (E) node[pos=0.5, left=6pt, scale = 0.8] {$\lambda_2$};
         \node[thick, BlueViolet] (F) at (-3, 1.5) {$\cdots$};
         \node[] (G) at (1.5, 1.5) {$\cdots$};
         \node[] (H) at (5.5, 1.5) {$\cdots$};
          \node[] (I) at (5.5, 2.2) {$\cdots$}; 
          \node[] (J) at (5.5, 3.1) {$\cdots$};
         \draw [decorate, decoration={brace, amplitude=5pt, mirror=false}] ([yshift=5pt]B.north west) -- ([yshift=5pt]C.north east) node [midway, above=6pt] {$\omega_{s}(1)$};
         \draw [decorate, decoration={brace, amplitude=5pt, mirror=false}] ([yshift=5pt]D.north west) -- ([yshift=5pt]E.north east) node [midway, above=6pt] {$\omega_{s}(2)$};
     \end{tikzpicture}
         \caption{Widths above an element $s$.}
         \label{fig:width}
    \end{figure}
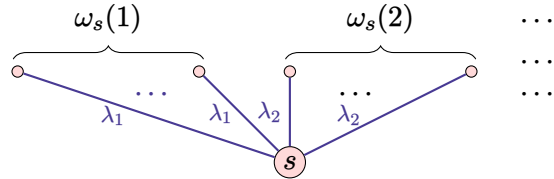
    
    \item Let $t$ cover $s$ in $\P$. We denote the \emph{index} of the label $\lambda(s, t)$ in $[\ell]$ by 
    \[
    \ind(s, t) \coloneqq \text{ the unique index $i \in [\ell]$ such that } \lambda_i =    \lambda(s, t). 
    \] 

    \item Let $\lambda_1' \vartriangleleft \cdots \vartriangleleft \lambda_{\ell'}'$ be the labels of the covering relations above $t$. We denote by $\Des_{s}(t)$ the set of indices $j$ in $[\ell']$ such that $\lambda'_j \vartriangleleft \lambda(s, t)$, that is, which gives rise to a descent, and we write $\des_{s}(t)$ for the cardinality of that set. 
    \begin{equation*}
        \Des_{s}(t) \coloneqq \{\: j \in [\ell']\mid \lambda_j' \vartriangleleft \lambda(s, t)\}, \quad \text{ and } \quad \des_{s}(t) \coloneqq \abs{\:\Des(s, t)\:}.  
    \end{equation*}
\end{itemize}
\end{definition}

The central object of study in this article is the following class of posets.

\begin{definition}[Rank-uniform labeled poset]
Let $(\P, \lambda)$ be an EL-labeled poset. We say that $(\P, \lambda)$ is \textit{rank-uniform} if the following holds.
\begin{itemize}
    \item (Uniform width) For all $s, s' \in \P$ of same rank and for all $i \geqslant 1$, the widths are equal: $\omega_s(i) = \omega_{s'}(i).$
    \item (Uniform descent) For all $s, s' \in \P$ of same rank and for every covering relations $s \lessdot t$ and $s' \lessdot t'$ such that $\ind(s,t) = \ind(s',t')$, the descent counts are equal: $\des_{s}(t) = \des_{s'}(t').$
\end{itemize}
\end{definition}

\begin{definition}[Monotonic labeled poset]
Let $(\P, \lambda)$ be a rank-uniform labeled poset, and for all $0 \leqslant k \leqslant \rk(\P)-1$ denote by $\ell_k$ the number of distinct labels above any element of rank $k$. The labeled poset $(\P, \lambda)$ is \textit{monotonic} if for all $0 \leqslant k \leqslant \rk(\P)-1$ the function $\des_k(\medbullet)$ is weakly increasing on $[\ell_k]$. A poset $\P$ admitting an EL-labeling $\lambda$ such that $(\P, \lambda)$ is a monotonic rank-uniform labeled poset is called \textit{uniformly monotonically edge-lexicographically shellable}, or \textit{UMEL-shellable} for short.
\end{definition}
In other words, a rank-uniform labeled poset $(\P, \lambda)$ is monotonic if at any element $s$ in $\P$, the larger the label above $s$, the more descents we observe from this label.

\section{Preliminary Lemmas}
In this section, we present several auxiliary lemmas that will be used throughout the subsequent sections. 

\begin{lemma}\label{lem:1}
Let $(P,\lambda)$ be an EL-labeled poset, let $s\in P$, and let $T$ be the set of elements covering $s$. If $t_1,t_2\in T$ and there exists an element $u\in P$ such that $t_1$ and $t_2$ are the only elements of $T$ covered by $u$, then \[ \lambda(s,t_1) \neq \lambda(s,t_2). \]
\end{lemma}

\begin{proof}
    The interval $[s,u]$ has exactly two maximal chains,
    \[s\lessdot t_1\lessdot u\qquad\text{and}\qquad s\lessdot t_2\lessdot u.\]
    Since $(P,\lambda)$ is EL-labeled, exactly one of these chains is increasing, while the other is decreasing. Suppose, for contradiction, that
    \[\lambda(s,t_1)=\lambda(s,t_2).\]
    Then the two chains have the same label on their first edge. Consequently, their lexicographic order is determined by the labels on their second edges. Since one chain is increasing and the other is decreasing, the decreasing chain has a smaller second label and is therefore lexicographically minimal. This contradicts the EL-labeling property, which requires the unique increasing maximal chain to be lexicographically minimal.
\end{proof}

\begin{lemma}\label{lem:2}
Let $(P,\lambda)$ be an EL-labeled poset, and let $T$ be the set of elements covering $s$. Then there exists a unique $t\in T$ such that \[ \operatorname{des}_s(t)=0. \] \end{lemma}

\begin{proof} Let $t\in T$ be such that $\lambda(s,t)$ is minimal among $\{\lambda(s,x):x\in T\}$. By an argument analogous to that in Lemma \ref{lem:1}, we have \[ \operatorname{des}_s(t)=0. \] Now suppose, for contradiction, that there exists another element $u\in T$ such that \[ \operatorname{des}_s(u)=0. \] By the EL-labeling property, the interval $[t,\hat{1}]$ has a unique increasing maximal chain \[ t\lessdot t_1 \lessdot t_2\lessdot\cdots\lessdot\hat{1}. \] Since $\operatorname{des}_s(t)=0$, the chain \[ s\lessdot t\lessdot t_1\lessdot t_2\lessdot\cdots\lessdot\hat{1} \] is increasing. Similarly, the interval $[u,\hat{1}]$ has an increasing maximal chain \[ u\lessdot u_1\lessdot u_2\lessdot\cdots\lessdot\hat{1}, \] and since $\operatorname{des}_s(u)=0$, the chain \[ s\lessdot u\lessdot u_1\lessdot u_2\lessdot\cdots\lessdot\hat{1} \] is also increasing. Thus, $[s,\hat{1}]$ contains two distinct increasing maximal chains, contradicting the uniqueness of the increasing maximal chain in an EL-labeling.
\end{proof}

\begin{corollary}\label{cor:1}
Let $(P,\lambda)$ be a labeled poset. Then, for every $s\in P$, \[ \omega_s(1)=1. \]
\end{corollary}

\section{Rank 4 counterexamples}\label{rank-4}

Let $Q$ be the bounded graded poset whose Hasse diagram is drawn in Figure \ref{fig:counter-example-1} below. It is straightforward to verify that $Q$ is rank-uniform. We will prove that $Q$ is EL-shellable but not UMEL-shellable.

\begin{figure}[H]
    \centering
    \begin{tikzpicture}[
        scale=0.75,
        baseline=14px
    ]

    \node[circle, draw=black, fill=red!15, inner sep=1.7pt]
        (one) at (0,5) {$\hat{1}$};

    \node[circle, draw=black, fill=red!15, inner sep=1.7pt]
        (z1) at (-4,3) {$z_1$};
    \node[circle, draw=black, fill=red!15, inner sep=1.7pt]
        (z2) at (-1.5,3) {$z_2$};
    \node[circle, draw=black, fill=red!15, inner sep=1.7pt]
        (z3) at (1.5,3) {$z_3$};
    \node[circle, draw=black, fill=red!15, inner sep=1.7pt]
        (z4) at (4,3) {$z_4$};

    \node[circle, draw=black, fill=red!15, inner sep=1.7pt]
        (y1) at (-4,1) {$y_1$};
    \node[circle, draw=black, fill=red!15, inner sep=1.7pt]
        (y2) at (-1.5,1) {$y_2$};
    \node[circle, draw=black, fill=red!15, inner sep=1.7pt]
        (y3) at (1.5,1) {$y_3$};
    \node[circle, draw=black, fill=red!15, inner sep=1.7pt]
        (y4) at (4,1) {$y_4$};

    \node[circle, draw=black, fill=red!15, inner sep=1.7pt]
        (x1) at (-1.5,-1) {$x_1$};
    \node[circle, draw=black, fill=red!15, inner sep=1.7pt]
        (x2) at (1.5,-1) {$x_2$};

    \node[circle, draw=black, fill=red!15, inner sep=1.7pt]
        (zero) at (0,-3) {$\hat{0}$};

    \draw[blue!70!black] (one) -- (z1);
    \draw[blue!70!black] (one) -- (z2);
    \draw[blue!70!black] (one) -- (z3);
    \draw[blue!70!black] (one) -- (z4);

    \draw[blue!70!black] (z1) -- (y1);

    \draw[blue!70!black] (z2) -- (y1);
    \draw[blue!70!black] (z2) -- (y2);
    \draw[blue!70!black] (z2) -- (y3);

    \draw[blue!70!black] (z3) -- (y2);
    \draw[blue!70!black] (z3) -- (y3);
    \draw[blue!70!black] (z3) -- (y4);

    \draw[blue!70!black] (z4) -- (y4);

    \draw[blue!70!black] (y1) -- (x1)
        node[midway, left=3pt] {$c$};

    \draw[blue!70!black] (y2) -- (x1)
        node[midway, left=4pt] {$d$};

    \draw[blue!70!black] (y3) -- (x1)
        node[midway, left=6pt] {$e$};

    \draw[blue!70!black] (y2) -- (x2)
        node[midway, right=6pt] {$f$};

    \draw[blue!70!black] (y3) -- (x2)
        node[midway, right=4pt] {$g$};

    \draw[blue!70!black] (y4) -- (x2)
        node[midway, right=4pt] {$h$};

    \draw[blue!70!black] (x1) -- (zero)
        node[midway, left=4pt] {$a$};

    \draw[blue!70!black] (x2) -- (zero)
        node[midway, right=4pt] {$b$};

    \draw[blue!70!black] (y3) -- (z2)
        node[midway, right=6pt] {$i$};

    \draw[blue!70!black] (y3) -- (z3)
        node[midway, right=4pt] {$j$};
        
    \end{tikzpicture}
    \caption{Hasse diagram of $Q$}
    \label{fig:counter-example-1}
\end{figure}
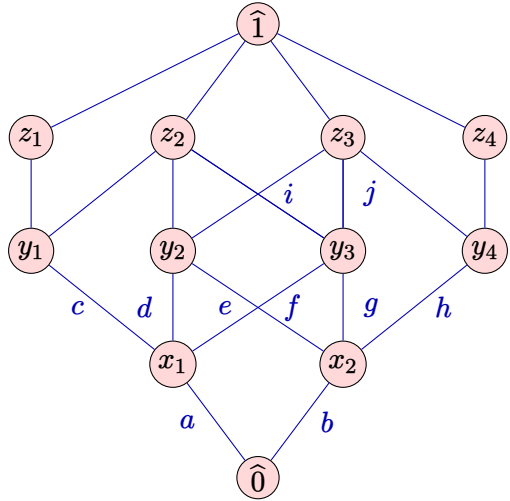

\begin{claim}
    \(Q\) is not UMEL-shellable.
\end{claim}

\begin{proof}
    Suppose that $\lambda$ is a labeling of $Q$ such that $(Q,\lambda)$ is a monotonic rank-uniform labeled poset. We denote the labels in Figure \ref{fig:counter-example-1} by $a,b,c,d,e,f,g,h,i,j$, respectively, representing the labeling $\lambda$. By Lemma \ref{lem:1} , we have $a\neq b$. By symmetry, we may assume without loss of generality that $a\vartriangleleft b$. In the interval \([\hat{0}, y_{2}]\) there are two maximal chains: \(\hat{0}\lessdot x_1\lessdot y_{2}\text{ and }\hat{0}\lessdot x_2\lessdot y_{2}.\) Since $a\vartriangleleft b$, the unique increasing chain must be the first. So, $a\trianglelefteqslant d$ and $f\vartriangleleft b$. Similarly, $a\trianglelefteqslant e$ and $g\vartriangleleft b$. Moreover, the interval $[\hat{0},y_4]$ contains only one maximal chain, which must therefore be increasing. Hence, \(b\trianglelefteqslant h\), which implies that \(g\trianglelefteqslant h\).\\
    
    By Lemma \ref{lem:1}, we have $f\neq g$, and we may assume without loss of generality that $f\vartriangleleft g$. By considering the intervals \([x_{2},z_{2}]\) and \([x_{2},z_{3}]\), we must have $i\vartriangleleft g$ and $j\vartriangleleft g$. Moreover, $i\neq j$ (Lemma \ref{lem:1}). Therefore, \(Des_{x_{2}}(y_{3}) = \{i,j\}\) and \(des_{x_{2}}(y_{3}) = 2\). On the other hand, the interval \([x_{2},z_{4}]\) has only one maximal chain, which implies that \(h \trianglelefteqslant \lambda(y_{4},z_{4})\). Hence, \(des_{x_{2}}(y_{4}) \leq 1 < des_{x_{2}}(y_{3})\). However, \(g\trianglelefteqslant h\), which contradicts the monotonicity condition. Hence \(Q\) is not UMEL-shellable.
\end{proof}

\begin{claim}
    \(Q\) is EL-shellable.
\end{claim}
\begin{proof}
    The Figure \ref{fig:counter-example-EL} below presents an EL-labeling of \(Q\) which can be verified by checking the intervals of \(Q\).
\end{proof}
    \begin{figure}[H]
        \centering
        \resizebox{0.5\textwidth}{!}{
        \begin{tikzpicture}[
            every node/.style={
                circle,
                draw=black,
                fill=red!15,
                inner sep=2pt,
                minimum size=7mm
            },
            edge/.style={
                draw=blue!70!black,
                thick
            },
            edge label/.style={
                draw=none,
                fill=white,
                inner sep=1.5pt,
                font=\small,
                text=blue!70!black
            }
        ]
    
    
        \node (one) at (0,5) {$\hat{1}$};
    
        \node (z1) at (-4,3) {$z_1$};
        \node (z2) at (-1.5,3) {$z_2$};
        \node (z3) at (1.5,3) {$z_3$};
        \node (z4) at (4,3) {$z_4$};
    
        \node (y1) at (-4,1) {$y_1$};
        \node (y2) at (-1.5,1) {$y_2$};
        \node (y3) at (1.5,1) {$y_3$};
        \node (y4) at (4,1) {$y_4$};
    
        \node (x1) at (-1.5,-1) {$x_1$};
        \node (x2) at (1.5,-1) {$x_2$};
    
        \node (zero) at (0,-3) {$\hat{0}$};
    
    
        \draw[edge]
            (one) -- (z1)
            node[pos=.48, above left=2pt, edge label] {$4$};
    
        \draw[edge]
            (one) -- (z2)
            node[pos=.75, above left=2pt, edge label] {$2$};
    
        \draw[edge]
            (one) -- (z3)
            node[pos=.75, above right=2pt, edge label] {$1$};
    
        \draw[edge]
            (one) -- (z4)
            node[pos=.48, above right=2pt, edge label] {$3$};
    
    
        \draw[edge]
            (z1) -- (y1)
            node[pos=.50, left=4pt, edge label] {$3$};
    
        \draw[edge]
            (z2) -- (y1)
            node[pos=.55, above left=2pt, edge label] {$4$};
    
        \draw[edge]
            (z2) -- (y2)
            node[pos=.50, left=2pt, edge label] {$2$};
    
        \draw[edge]
            (z2) -- (y3)
            node[pos=.75, above right=1pt, edge label] {$1$};
    
        \draw[edge]
            (z3) -- (y2)
            node[pos=.75, above left=1pt, edge label] {$3$};
    
        \draw[edge]
            (z3) -- (y3)
            node[pos=.50, right=1pt, edge label] {$2$};
    
        \draw[edge]
            (z3) -- (y4)
            node[pos=.65, above right=3pt, edge label] {$1$};
    
        \draw[edge]
            (z4) -- (y4)
            node[pos=.50, right=4pt, edge label] {$4$};
    
    
        \draw[edge]
            (y1) -- (x1)
            node[pos=.45, below left=3pt, edge label] {$2$};
    
        \draw[edge]
            (y2) -- (x1)
            node[pos=.50, left=2pt, edge label] {$3$};
    
        \draw[edge]
            (y2) -- (x2)
            node[pos=.65, right=8pt, edge label] {$2$};
    
        \draw[edge]
            (y3) -- (x1)
            node[pos=.65, left=8pt, edge label] {$4$};
    
        \draw[edge]
            (y3) -- (x2)
            node[pos=.50, right=2pt, edge label] {$3$};
    
        \draw[edge]
            (y4) -- (x2)
            node[pos=.50, below right=3pt, edge label] {$4$};
    
    
        \draw[edge]
            (x1) -- (zero)
            node[pos=.48, left=4pt, edge label] {$1$};
    
        \draw[edge]
            (x2) -- (zero)
            node[pos=.48, right=4pt, edge label] {$4$};
    
        \end{tikzpicture}}
        \caption{EL-labeling of $Q$}
        \label{fig:counter-example-EL}
    \end{figure}
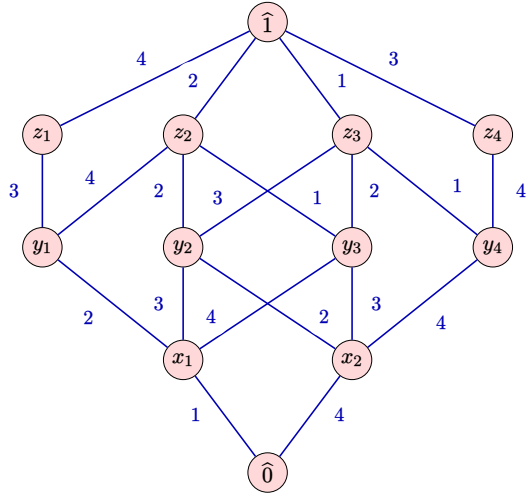
It is also not difficult to find a rank-uniform EL-labeling of $Q$. In Figure \ref{fig:counter-example-2} below, we present a poset $R$ that is EL-shellable but does not admit a rank-uniform labeling.

\begin{figure}[H]
    \centering
        \resizebox{0.55\textwidth}{!}{
    \begin{tikzpicture}[
        every node/.style={
            circle,
            draw=black,
            fill=red!15,
            inner sep=2pt,
            minimum size=7mm
        },
        edge/.style={
            draw=blue!70!black,
            thick
        }
    ]


    \node (one) at (0,5) {$\hat{1}$};

    \node (z1) at (-5,3) {$z_1$};
    \node (z2) at (-2.5,3) {$z_2$};
    \node (z3) at (0,3) {$z_3$};
    \node (z4) at (2.5,3) {$z_4$};
    \node (z5) at (5,3) {$z_5$};

    \node (y1) at (-3.7,1) {$y_1$};
    \node (y2) at (-1.3,1) {$y_2$};
    \node (y3) at (1.3,1) {$y_3$};
    \node (y4) at (3.7,1) {$y_4$};

    \node (x1) at (-2,-1) {$x_1$};
    \node (x2) at (0,-1) {$x_2$};
    \node (x3) at (2,-1) {$x_3$};

    \node (zero) at (0,-3) {$\hat{0}$};


    \draw[edge] (one) -- (z1);
    \draw[edge] (one) -- (z2);
    \draw[edge] (one) -- (z3);
    \draw[edge] (one) -- (z4);
    \draw[edge] (one) -- (z5);


    \draw[edge] (z1) -- (y1);
    \draw[edge] (z1) -- (y2);

    \draw[edge] (z2) -- (y1);
    \draw[edge] (z2) -- (y4);

    \draw[edge] (z3) -- (y1);
    \draw[edge] (z3) -- (y2);
    \draw[edge] (z3) -- (y3);
    \draw[edge] (z3) -- (y4);

    \draw[edge] (z4) -- (y2);
    \draw[edge] (z4) -- (y3);
    \draw[edge] (z4) -- (y4);

    \draw[edge] (z5) -- (y3);


    \draw[edge] (y1) -- (x1);
    \draw[edge] (y1) -- (x2);

    \draw[edge] (y2) -- (x1);
    \draw[edge] (y2) -- (x3);

    \draw[edge] (y3) -- (x1);
    \draw[edge] (y3) -- (x2);
    \draw[edge] (y3) -- (x3);

    \draw[edge] (y4) -- (x2);
    \draw[edge] (y4) -- (x3);


    \draw[edge] (x1) -- (zero);
    \draw[edge] (x2) -- (zero);
    \draw[edge] (x3) -- (zero);
    \end{tikzpicture}}
    \caption{Hasse diagram of $R$}
    \label{fig:counter-example-2}
\end{figure}
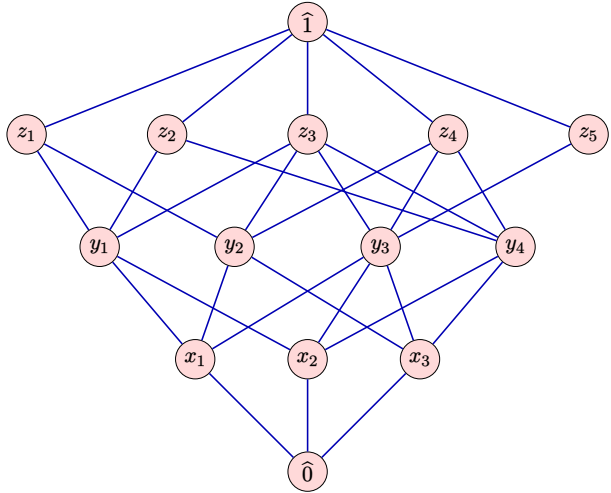
It can be verified by routine checks that $R$ is rank-uniform. To provide a clearer picture of $R$, we include the Hasse diagrams of the three intervals $[x_1,\hat{1}]$, $[x_2,\hat{1}]$, and $[x_3,\hat{1}]$ in the figures below.

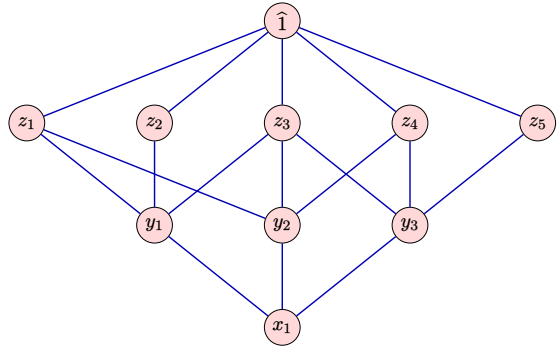
\begin{figure}[H]
    \centering
        \resizebox{0.5\textwidth}{!}{
    \begin{tikzpicture}[
        every node/.style={
            circle,
            draw=black,
            fill=red!15,
            inner sep=2pt,
            minimum size=7mm
        },
        edge/.style={
            draw=blue!70!black,
            thick
        }
    ]


    \node (one) at (0,5) {$\hat{1}$};

    \node (z1) at (-5,3) {$z_1$};
    \node (z2) at (-2.5,3) {$z_2$};
    \node (z3) at (0,3) {$z_3$};
    \node (z4) at (2.5,3) {$z_4$};
    \node (z5) at (5,3) {$z_5$};

    \node (y1) at (-2.5,1) {$y_1$};
    \node (y2) at (0,1) {$y_2$};
    \node (y3) at (2.5,1) {$y_3$};

    \node (x1) at (0,-1) {$x_1$};


    \draw[edge] (one) -- (z1);
    \draw[edge] (one) -- (z2);
    \draw[edge] (one) -- (z3);
    \draw[edge] (one) -- (z4);
    \draw[edge] (one) -- (z5);


    \draw[edge] (z1) -- (y1);
    \draw[edge] (z1) -- (y2);

    \draw[edge] (z2) -- (y1);

    \draw[edge] (z3) -- (y1);
    \draw[edge] (z3) -- (y2);
    \draw[edge] (z3) -- (y3);

    \draw[edge] (z4) -- (y2);
    \draw[edge] (z4) -- (y3);

    \draw[edge] (z5) -- (y3);


    \draw[edge] (y1) -- (x1);
    \draw[edge] (y2) -- (x1);
    \draw[edge] (y3) -- (x1);
    \end{tikzpicture}}
    \caption{Hasse diagram of $[x_1,\hat{1}]$}
    \label{fig:x-1}
\end{figure}
\begin{figure}[H]
    \centering
        \resizebox{0.5\textwidth}{!}{
    \begin{tikzpicture}[
        every node/.style={
            circle,
            draw=black,
            fill=red!15,
            inner sep=2pt,
            minimum size=7mm
        },
        edge/.style={
            draw=blue!70!black,
            thick
        }
    ]


    \node (one) at (0,5) {$\hat{1}$};

    \node (z1) at (-5,3) {$z_1$};
    \node (z2) at (-2.5,3) {$z_2$};
    \node (z3) at (0,3) {$z_3$};
    \node (z4) at (2.5,3) {$z_4$};
    \node (z5) at (5,3) {$z_5$};

    \node (y1) at (-2.5,1) {$y_1$};
    \node (y3) at (0,1) {$y_3$};
    \node (y4) at (2.5,1) {$y_4$};

    \node (x2) at (0,-1) {$x_2$};


    \draw[edge] (one) -- (z1);
    \draw[edge] (one) -- (z2);
    \draw[edge] (one) -- (z3);
    \draw[edge] (one) -- (z4);
    \draw[edge] (one) -- (z5);


    \draw[edge] (z1) -- (y1);

    \draw[edge] (z2) -- (y1);
    \draw[edge] (z2) -- (y4);

    \draw[edge] (z3) -- (y1);
    \draw[edge] (z3) -- (y3);
    \draw[edge] (z3) -- (y4);

    \draw[edge] (z4) -- (y3);
    \draw[edge] (z4) -- (y4);

    \draw[edge] (z5) -- (y3);


    \draw[edge] (y1) -- (x2);
    \draw[edge] (y3) -- (x2);
    \draw[edge] (y4) -- (x2);

    \end{tikzpicture}}
    \caption{Hasse diagram of $[x_2,\hat{1}]$}
    \label{fig:x-2}
\end{figure}
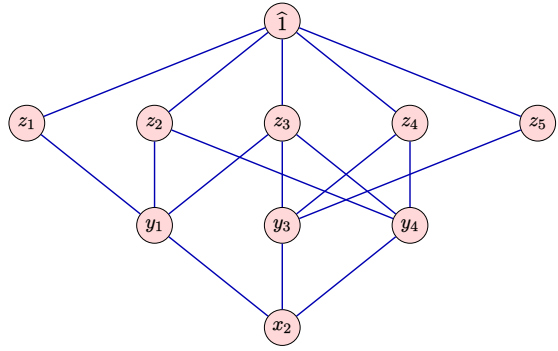
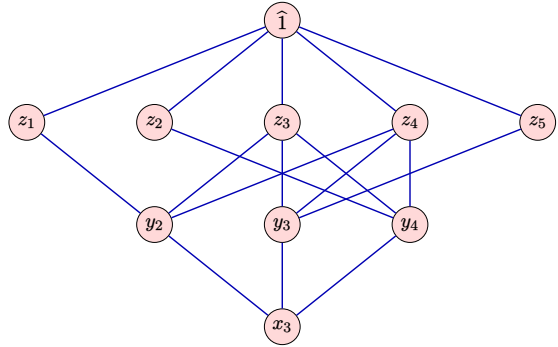
\begin{figure}[H]
    \centering
        \resizebox{0.5\textwidth}{!}{
    \begin{tikzpicture}[
        every node/.style={
            circle,
            draw=black,
            fill=red!15,
            inner sep=2pt,
            minimum size=7mm
        },
        edge/.style={
            draw=blue!70!black,
            thick
        }
    ]


    \node (one) at (0,5) {$\hat{1}$};

    \node (z1) at (-5,3) {$z_1$};
    \node (z2) at (-2.5,3) {$z_2$};
    \node (z3) at (0,3) {$z_3$};
    \node (z4) at (2.5,3) {$z_4$};
    \node (z5) at (5,3) {$z_5$};

    \node (y2) at (-2.5,1) {$y_2$};
    \node (y3) at (0,1) {$y_3$};
    \node (y4) at (2.5,1) {$y_4$};

    \node (x3) at (0,-1) {$x_3$};


    \draw[edge] (one) -- (z1);
    \draw[edge] (one) -- (z2);
    \draw[edge] (one) -- (z3);
    \draw[edge] (one) -- (z4);
    \draw[edge] (one) -- (z5);


    \draw[edge] (z1) -- (y2);

    \draw[edge] (z2) -- (y4);

    \draw[edge] (z3) -- (y2);
    \draw[edge] (z3) -- (y3);
    \draw[edge] (z3) -- (y4);

    \draw[edge] (z4) -- (y2);
    \draw[edge] (z4) -- (y3);
    \draw[edge] (z4) -- (y4);

    \draw[edge] (z5) -- (y3);


    \draw[edge] (y2) -- (x3);
    \draw[edge] (y3) -- (x3);
    \draw[edge] (y4) -- (x3);

    \end{tikzpicture}}
    \caption{Hasse diagram of $[x_3,\hat{1}]$}
    \label{fig:x-3}
\end{figure}
Suppose that \(\lambda\) is an EL-labeling of \(R\), such that \((R,\lambda)\) is a rank-uniform labeled poset. We first focus on the interval $x_3,\hat{1}$ shown in the figure.
\begin{claim}
    For all $1\leq i\leq 4$,
\[
\omega_{y_i}(1)=\omega_{y_i}(2)=\omega_{y_i}(3)=1.
\]
i.e., all elements above $y_i$ are connected by a unique label, and hence all corresponding labels are distinct.
\end{claim}
\begin{proof}
Since we have assumed that $(R,\lambda)$ is a rank-uniform labeled poset, it is enough to prove the claim for one element in rank $2$, as the result for the remaining elements follows from the uniformity of the widths. By Corollary~\ref{cor:1}, we have $\omega_{x_3}(1)=1$. Without loss of generality, assume that the minimum label above $x_3$ is $\lambda(x_3,y_2)$. Then the unique increasing maximal chains in the intervals $[x_3,z_1]$, $[x_3,z_3]$, and $[x_3,z_4]$ are
\[
x_3\lessdot y_2\lessdot z_1,\qquad
x_3\lessdot y_2\lessdot z_3,\qquad
x_3\lessdot y_2\lessdot z_4,
\]
respectively. Moreover, the unique increasing maximal chains in the intervals $[x_3,z_2]$ and $[x_3,z_5]$ are forced to be
\[
x_3\lessdot y_4\lessdot z_2
\qquad\text{and}\qquad
x_3\lessdot y_3\lessdot z_5,
\]
respectively. Thus, by the EL property, we have
\[
\lambda(y_4,z_3),\lambda(y_4,z_2)
\trianglelefteqslant \lambda(x_3,y_4),
\]
and, since the chain
\[
x_3\lessdot y_4\lessdot z_2
\]
is increasing,
\[
\lambda(x_3,y_4)\vartriangleleft\lambda(y_4,z_2).
\]
By Corollary~\ref{cor:1}, we also have
\[
\lambda(y_4,z_3)\neq\lambda(y_4,z_2).
\]
It follows that
\[
\omega_{y_4}(1)=\omega_{y_4}(2)=\omega_{y_4}(3)=1.
\]
By the uniformity of the widths, the same holds for all $y_i$ with
$1\leq i\leq 4$.
\end{proof}

\begin{claim}
    For all $1\leq i\leq 3$ and $1\leq j\leq 4$, there exists some
$1\leq k\leq 5$ such that
\[
x_i\lessdot y_j\lessdot z_k.
\qquad\text{and}\qquad
\lambda(x_{i},y_{j})\trianglelefteqslant\lambda(y_{j},z_{k}) 
\]
\end{claim}
\begin{proof}
    In the previous claim, we showed that
    \[
    \bigl(\operatorname{des}_{x_3}(y_2),
    \operatorname{des}_{x_3}(y_3),
    \operatorname{des}_{x_3}(y_4)\bigr)
    \]
    is a permutation of $(0,2,2)$. Moreover, we proved that the labels above
    each element of rank $2$ are distinct. Therefore, there is an ascent
    above each of the edges $x_3\lessdot y_j$, for $j=2,3,4$, which proves
    the claim for $i=3$. The result for $i=1,2$ follows from the uniformity
    of the descents.
\end{proof}

\begin{claim}
    \((R,\lambda)\) cannot be a rank-uniform labeled poset.
\end{claim}
\begin{proof}
    By Lemma~\ref{lem:1}, we know that
$\lambda(\hat{0},x_1)$, $\lambda(\hat{0},x_2)$, and
$\lambda(\hat{0},x_3)$ are distinct. Let
$\lambda(\hat{0},x_{i_1})$ be the minimum of these three labels.

There exists an element $y_j$ that does not cover $x_{i_1}$ but covers
one of the other two elements, say $x_{i_2}$, where $i_2\neq i_1$.
Since $y_j$ does not cover $x_{i_1}$, the unique increasing maximal
chain in the interval $[\hat{0},y_j]$ must pass through one of the other
two elements. Without loss of generality, suppose it passes through
$x_{i_2}$. Thus,
\[
\lambda(\hat{0},x_{i_2})
\trianglelefteqslant
\lambda(x_{i_2},y_j).
\]

By the previous claim, there exists $k$ such that
\[
x_{i_2}\lessdot y_j\lessdot z_k
\qquad\text{and}\qquad
\lambda(x_{i_2},y_j)
\trianglelefteqslant
\lambda(y_j,z_k).
\]
Therefore,
\[
\hat{0}\lessdot x_{i_2}\lessdot y_j\lessdot z_k
\]
is an increasing maximal chain in the interval $[\hat{0},z_k]$, with
\[
\lambda(\hat{0},x_{i_2})
\trianglelefteqslant
\lambda(x_{i_2},y_j)
\trianglelefteqslant
\lambda(y_j,z_k).
\]

However, since $\lambda(\hat{0},x_{i_1})$ is the minimum of the three
labels,
\[
\lambda(\hat{0},x_{i_1})
\vartriangleleft
\lambda(\hat{0},x_{i_2}).
\]
Moreover, $x_{i_1}\lessdot y_{j'}\lessdot z_k$ for some $y_{j'}$, so
there is a maximal chain in $[\hat{0},z_k]$ whose first label is
$\lambda(\hat{0},x_{i_1})$. Hence, this chain is lexicographically
smaller than the increasing maximal chain through $x_{i_2}$. This
contradicts the EL-labeling property, which requires the unique
increasing maximal chain to be lexicographically minimal.
\end{proof}
\begin{claim}
    \(R\) is EL-shellable.
\end{claim}
\begin{proof}
Below, we provide an EL-labeling of $R$, together with the induced labelings on the intervals $[x_1,\hat{1}]$, $[x_2,\hat{1}]$, and $[x_3,\hat{1}]$, to give a clearer picture of the labeling.
    \begin{figure}[H]
    \centering
        \resizebox{0.6\textwidth}{!}{
    \begin{tikzpicture}[
        every node/.style={
            circle,
            draw=black,
            fill=red!15,
            inner sep=2pt,
            minimum size=7mm
        },
        edge/.style={
            draw=blue!70!black,
            thick
        },
        edge label/.style={
            draw=none,
            fill=white,
            inner sep=1.2pt,
            font=\small,
            text=blue!70!black
        }
    ]


    \node (one) at (0,5) {$\hat{1}$};

    \node (z1) at (-5,3) {$z_1$};
    \node (z2) at (-2.5,3) {$z_2$};
    \node (z3) at (0,3) {$z_3$};
    \node (z4) at (2.5,3) {$z_4$};
    \node (z5) at (5,3) {$z_5$};

    \node (y1) at (-3.7,1) {$y_1$};
    \node (y2) at (-1.3,1) {$y_2$};
    \node (y3) at (1.3,1) {$y_3$};
    \node (y4) at (3.7,1) {$y_4$};

    \node (x1) at (-2,-1) {$x_1$};
    \node (x2) at (0,-1) {$x_2$};
    \node (x3) at (2,-1) {$x_3$};

    \node (zero) at (0,-3) {$\hat{0}$};


    \draw[edge] (one) -- (z1)
        node[pos=.55, above left=3pt, edge label] {$1$};

    \draw[edge] (one) -- (z2)
        node[pos=.55, above left=3pt, edge label] {$1$};

    \draw[edge] (one) -- (z3)
        node[pos=.52, left=5pt, edge label] {$1$};

    \draw[edge] (one) -- (z4)
        node[pos=.55, above right=3pt, edge label] {$1$};

    \draw[edge] (one) -- (z5)
        node[pos=.55, above right=3pt, edge label] {$1$};


    \draw[edge] (z1) -- (y1)
        node[pos=.63, left=3pt, edge label] {$1$};

    \draw[edge] (z1) -- (y2)
        node[pos=.95, left=2pt, edge label] {$1$};

    \draw[edge] (z2) -- (y1)
        node[pos=.62, left=3pt, edge label] {$2$};

    \draw[edge] (z2) -- (y4)
        node[pos=.96, left=2pt, edge label] {$3$};

    \draw[edge] (z3) -- (y1)
        node[pos=.95, right=0pt, edge label] {$4$};

    \draw[edge] (z3) -- (y2)
        node[pos=.62, left=3pt, edge label] {$2$};

    \draw[edge] (z3) -- (y3)
        node[pos=.92, left=3pt, edge label] {$1$};

    \draw[edge] (z3) -- (y4)
        node[pos=.92, above=2pt, edge label] {$2$};

    \draw[edge] (z4) -- (y2)
        node[pos=.88, right=2pt, edge label] {$4$};

    \draw[edge] (z4) -- (y3)
        node[pos=.74, left=3pt, edge label] {$3$};

    \draw[edge] (z4) -- (y4)
        node[pos=.62, right=3pt, edge label] {$1$};

    \draw[edge] (z5) -- (y3)
        node[pos=.78, left=2pt, edge label] {$4$};


    \draw[edge] (y1) -- (x1)
        node[pos=.62, left=3pt, edge label] {$1$};

    \draw[edge] (y1) -- (x2)
        node[pos=.82, below=2pt, edge label] {$1$};

    \draw[edge] (y2) -- (x1)
        node[pos=.62, left=3pt, edge label] {$3$};

    \draw[edge] (y2) -- (x3)
        node[pos=.78, below=2pt, edge label] {$1$};

    \draw[edge] (y3) -- (x1)
        node[pos=.92, above=2pt, edge label] {$4$};

    \draw[edge] (y3) -- (x2)
        node[pos=.62, left=3pt, edge label] {$3$};

    \draw[edge] (y3) -- (x3)
        node[pos=.62, right=3pt, edge label] {$4$};

    \draw[edge] (y4) -- (x2)
        node[pos=.82, above=2pt, edge label] {$4$};

    \draw[edge] (y4) -- (x3)
        node[pos=.62, right=3pt, edge label] {$3$};


    \draw[edge] (x1) -- (zero)
        node[pos=.52, left=4pt, edge label] {$1$};

    \draw[edge] (x2) -- (zero)
        node[pos=.52, left=4pt, edge label] {$4$};

    \draw[edge] (x3) -- (zero)
        node[pos=.52, right=4pt, edge label] {$6$};

    \end{tikzpicture}}
    \caption{El-labeling of \(R\)}
\end{figure}
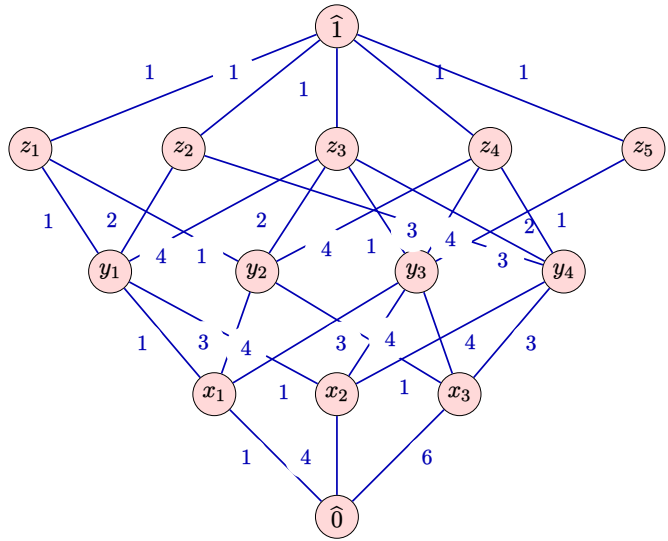
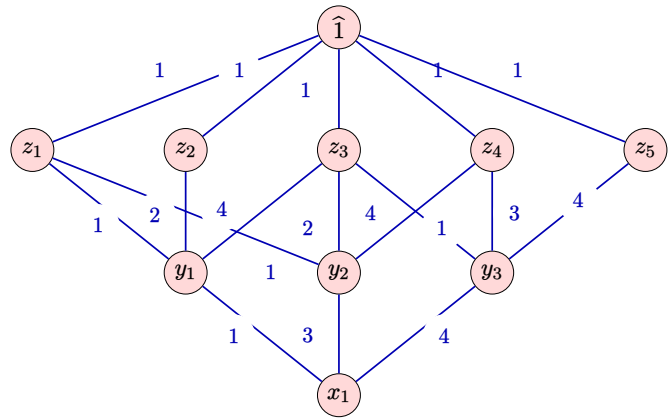
\begin{figure}[H]
    \centering
        \resizebox{0.6\textwidth}{!}{
    \begin{tikzpicture}[
        every node/.style={
            circle,
            draw=black,
            fill=red!15,
            inner sep=2pt,
            minimum size=7mm
        },
        edge/.style={
            draw=blue!70!black,
            thick
        },
        edge label/.style={
            draw=none,
            fill=white,
            inner sep=1.2pt,
            font=\small,
            text=blue!70!black
        }
    ]


    \node (one) at (0,5) {$\hat{1}$};

    \node (z1) at (-5,3) {$z_1$};
    \node (z2) at (-2.5,3) {$z_2$};
    \node (z3) at (0,3) {$z_3$};
    \node (z4) at (2.5,3) {$z_4$};
    \node (z5) at (5,3) {$z_5$};

    \node (y1) at (-2.5,1) {$y_1$};
    \node (y2) at (0,1) {$y_2$};
    \node (y3) at (2.5,1) {$y_3$};

    \node (x1) at (0,-1) {$x_1$};


    \draw[edge] (one) -- (z1)
        node[pos=.52, above left=3pt, edge label] {$1$};

    \draw[edge] (one) -- (z2)
        node[pos=.52, above left=3pt, edge label] {$1$};

    \draw[edge] (one) -- (z3)
        node[pos=.52, left=5pt, edge label] {$1$};

    \draw[edge] (one) -- (z4)
        node[pos=.52, above right=3pt, edge label] {$1$};

    \draw[edge] (one) -- (z5)
        node[pos=.52, above right=3pt, edge label] {$1$};


    \draw[edge] (z1) -- (y1)
        node[pos=.65, left=3pt, edge label] {$1$};

    \draw[edge] (z1) -- (y2)
        node[pos=.82, below=2pt, edge label] {$1$};

    \draw[edge] (z2) -- (y1)
        node[pos=.55, left=4pt, edge label] {$2$};

    \draw[edge] (z3) -- (y1)
        node[pos=.68, above left=2pt, edge label] {$4$};

    \draw[edge] (z3) -- (y2)
        node[pos=.72, left=4pt, edge label] {$2$};

    \draw[edge] (z3) -- (y3)
        node[pos=.88, above left=2pt, edge label] {$1$};

    \draw[edge] (z4) -- (y2)
        node[pos=.72, above left=2pt, edge label] {$4$};

    \draw[edge] (z4) -- (y3)
        node[pos=.52, right=0pt, edge label] {$3$};

    \draw[edge] (z5) -- (y3)
        node[pos=.58, above right=2pt, edge label] {$4$};


    \draw[edge] (y1) -- (x1)
        node[pos=.52, left=4pt, edge label] {$1$};

    \draw[edge] (y2) -- (x1)
        node[pos=.52, left=4pt, edge label] {$3$};

    \draw[edge] (y3) -- (x1)
        node[pos=.52, right=4pt, edge label] {$4$};

    \end{tikzpicture}}
    \caption{El-labeling of \([x_{1},\hat{1}]\)}
\end{figure}
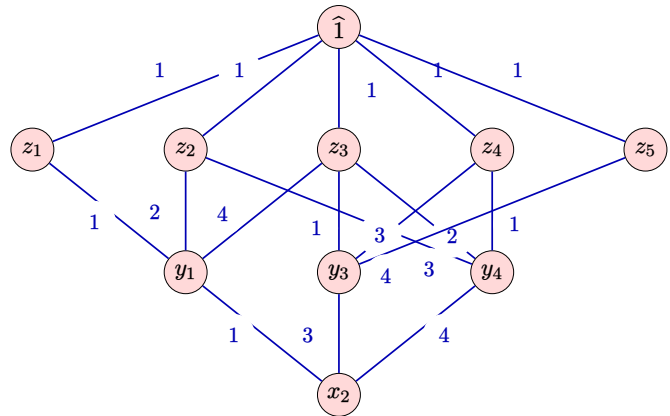
\begin{figure}[H]
    \centering
        \resizebox{0.6\textwidth}{!}{
    \begin{tikzpicture}[
        every node/.style={
            circle,
            draw=black,
            fill=red!15,
            inner sep=2pt,
            minimum size=7mm
        },
        edge/.style={
            draw=blue!70!black,
            thick
        },
        edge label/.style={
            draw=none,
            fill=white,
            inner sep=1.2pt,
            font=\small,
            text=blue!70!black
        }
    ]


    \node (one) at (0,5) {$\hat{1}$};

    \node (z1) at (-5,3) {$z_1$};
    \node (z2) at (-2.5,3) {$z_2$};
    \node (z3) at (0,3) {$z_3$};
    \node (z4) at (2.5,3) {$z_4$};
    \node (z5) at (5,3) {$z_5$};

    \node (y1) at (-2.5,1) {$y_1$};
    \node (y3) at (0,1) {$y_3$};
    \node (y4) at (2.5,1) {$y_4$};

    \node (x2) at (0,-1) {$x_2$};


    \draw[edge] (one) -- (z1)
        node[pos=.52, above left=3pt, edge label] {$1$};

    \draw[edge] (one) -- (z2)
        node[pos=.52, above left=3pt, edge label] {$1$};

    \draw[edge] (one) -- (z3)
        node[pos=.52, right=5pt, edge label] {$1$};

    \draw[edge] (one) -- (z4)
        node[pos=.52, above right=3pt, edge label] {$1$};

    \draw[edge] (one) -- (z5)
        node[pos=.52, above right=3pt, edge label] {$1$};


    \draw[edge] (z1) -- (y1)
        node[pos=.62, left=3pt, edge label] {$1$};

    \draw[edge] (z2) -- (y1)
        node[pos=.52, left=4pt, edge label] {$2$};

    \draw[edge] (z2) -- (y4)
        node[pos=.84, below=0pt, edge label] {$3$};

    \draw[edge] (z3) -- (y1)
        node[pos=.70, above left=0pt, edge label] {$4$};

    \draw[edge] (z3) -- (y3)
        node[pos=.72, left=0pt, edge label] {$1$};

    \draw[edge] (z3) -- (y4)
        node[pos=.98, above left=3pt, edge label] {$2$};

    \draw[edge] (z4) -- (y3)
        node[pos=.97, above right=3pt, edge label] {$3$};

    \draw[edge] (z4) -- (y4)
        node[pos=.67, right=0pt, edge label] {$1$};

    \draw[edge] (z5) -- (y3)
        node[pos=.90, below=0pt, edge label] {$4$};


    \draw[edge] (y1) -- (x2)
        node[pos=.52, left=4pt, edge label] {$1$};

    \draw[edge] (y3) -- (x2)
        node[pos=.52, left=4pt, edge label] {$3$};

    \draw[edge] (y4) -- (x2)
        node[pos=.52, right=4pt, edge label] {$4$};

    \end{tikzpicture}}
    \caption{El-labeling of \([x_{2},\hat{1}]\)}
\end{figure}
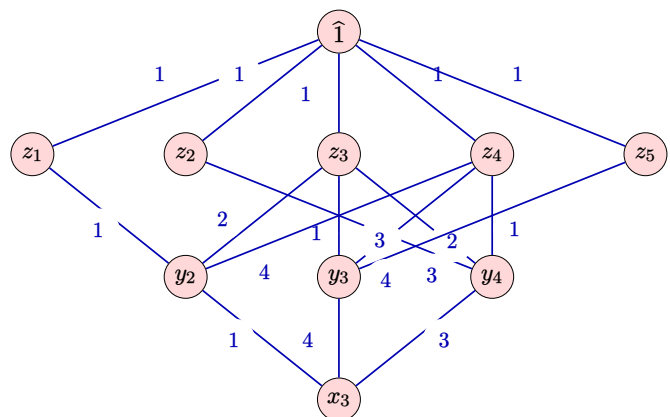
\begin{figure}[H]
    \centering
        \resizebox{0.6\textwidth}{!}{
    \begin{tikzpicture}[
        every node/.style={
            circle,
            draw=black,
            fill=red!15,
            inner sep=2pt,
            minimum size=7mm
        },
        edge/.style={
            draw=blue!70!black,
            thick
        },
        edge label/.style={
            draw=none,
            fill=white,
            inner sep=1.2pt,
            font=\small,
            text=blue!70!black
        }
    ]


    \node (one) at (0,5) {$\hat{1}$};

    \node (z1) at (-5,3) {$z_1$};
    \node (z2) at (-2.5,3) {$z_2$};
    \node (z3) at (0,3) {$z_3$};
    \node (z4) at (2.5,3) {$z_4$};
    \node (z5) at (5,3) {$z_5$};

    \node (y2) at (-2.5,1) {$y_2$};
    \node (y3) at (0,1) {$y_3$};
    \node (y4) at (2.5,1) {$y_4$};

    \node (x3) at (0,-1) {$x_3$};


    \draw[edge] (one) -- (z1)
        node[pos=.52, above left=3pt, edge label] {$1$};

    \draw[edge] (one) -- (z2)
        node[pos=.52, above left=3pt, edge label] {$1$};

    \draw[edge] (one) -- (z3)
        node[pos=.52, left=5pt, edge label] {$1$};

    \draw[edge] (one) -- (z4)
        node[pos=.52, above right=3pt, edge label] {$1$};

    \draw[edge] (one) -- (z5)
        node[pos=.52, above right=3pt, edge label] {$1$};


    \draw[edge] (z1) -- (y2)
        node[pos=.65, left=3pt, edge label] {$1$};

    \draw[edge] (z2) -- (y4)
        node[pos=.85, below=0pt, edge label] {$3$};

    \draw[edge] (z3) -- (y2)
        node[pos=.70, above left=0pt, edge label] {$2$};

    \draw[edge] (z3) -- (y3)
        node[pos=.72, left=0pt, edge label] {$1$};

    \draw[edge] (z3) -- (y4)
        node[pos=.98, above left=3pt, edge label] {$2$};

    \draw[edge] (z4) -- (y2)
        node[pos=.78, below=2pt, edge label] {$4$};
    \draw[edge] (z4) -- (y3)
        node[pos=.97, above right=3pt, edge label] {$3$};

    \draw[edge] (z4) -- (y4)
        node[pos=.67, right=0pt, edge label] {$1$};

    \draw[edge] (z5) -- (y3)
        node[pos=.90, below=0pt, edge label] {$4$};


    \draw[edge] (y2) -- (x3)
        node[pos=.52, left=4pt, edge label] {$1$};

    \draw[edge] (y3) -- (x3)
        node[pos=.52, left=4pt, edge label] {$4$};

    \draw[edge] (y4) -- (x3)
        node[pos=.52, right=4pt, edge label] {$3$};

    \end{tikzpicture}}
    \caption{El-labeling of \([x_{3},\hat{1}]\)}
\end{figure}
\end{proof}
\section{Rank 3 counterexamples}
Let $S$ be the bounded graded poset whose Hasse diagram is drawn in Figure \ref{fig:counter-example-3} below. It is straightforward to verify that $S$ is rank-uniform. We will prove that $S$ is EL-shellable but not UMEL-shellable.
\begin{figure}[H]
    \centering
        \resizebox{0.8\textwidth}{!}{
    \begin{tikzpicture}[
        every node/.style={
            circle,
            draw=black,
            fill=red!15,
            inner sep=2pt,
            minimum size=7mm
        },
        edge/.style={
            draw=blue!70!black,
            thick
        }
    ]


    \node (one) at (0,5) {$\hat{1}$};

    \node (y1) at (-7,3) {$y_1$};
    \node (y2) at (-5,3) {$y_2$};
    \node (y3) at (-3,3) {$y_3$};
    \node (y4) at (-1,3) {$y_4$};
    \node (y5) at (1,3) {$y_5$};
    \node (y6) at (3,3) {$y_6$};
    \node (y7) at (5,3) {$y_7$};
    \node (y8) at (7,3) {$y_8$};

    \node (x1) at (-5,1) {$x_1$};
    \node (x2) at (-1.7,1) {$x_2$};
    \node (x3) at (1.7,1) {$x_3$};
    \node (x4) at (5,1) {$x_4$};

    \node (zero) at (0,-1.5) {$\hat{0}$};


    \draw[edge] (one) -- (y1);
    \draw[edge] (one) -- (y2);
    \draw[edge] (one) -- (y3);
    \draw[edge] (one) -- (y4);
    \draw[edge] (one) -- (y5);
    \draw[edge] (one) -- (y6);
    \draw[edge] (one) -- (y7);
    \draw[edge] (one) -- (y8);


    \draw[edge] (y1) -- (x1);
    \draw[edge] (y2) -- (x1);
    \draw[edge] (y3) -- (x1);

    \draw[edge] (y3) -- (x2);
    \draw[edge] (y4) -- (x2);
    \draw[edge] (y5) -- (x2);

    \draw[edge] (y4) -- (x3);
    \draw[edge] (y5) -- (x3);
    \draw[edge] (y6) -- (x3);

    \draw[edge] (y6) -- (x4);
    \draw[edge] (y7) -- (x4);
    \draw[edge] (y8) -- (x4);


    \draw[edge] (x1) -- (zero);
    \draw[edge] (x2) -- (zero);
    \draw[edge] (x3) -- (zero);
    \draw[edge] (x4) -- (zero);

    \end{tikzpicture}}
    \caption{Hasse diagram of $S$}
    \label{fig:counter-example-3}
\end{figure}
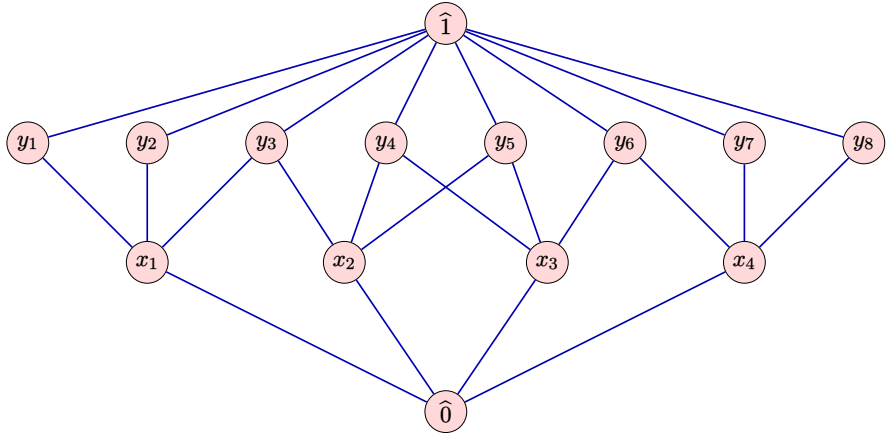
\begin{claim}
    \(S\) is not UMEL-shellable
\end{claim}
\begin{proof}
    Suppose that $\lambda$ is a labeling of $S$ such that $(S,\lambda)$ is a
    monotonic rank-uniform labeled poset. By Corollary~\ref{cor:1}, we have
    \[
    \omega_{\hat{0}}(1)=1.
    \]
    Without loss of generality, we may assume that the minimum label on an edge
    covering $\hat{0}$ is either
    \[
    \lambda(\hat{0},x_1)
    \qquad\text{or}\qquad
    \lambda(\hat{0},x_2).
    \]
    
    Since $\hat{0}\lessdot x_4\lessdot y_7$ and
    $\hat{0}\lessdot x_4\lessdot y_8$ are forced to be increasing, if
    $\hat{0}\lessdot x_4\lessdot y_6$ were also increasing, then
    $\des_{\hat{0}}(x_4)=0$, contradicting Lemma~\ref{lem:2}.
    
    Therefore, in the interval $[\hat{0},y_6]$, the chain
    $\hat{0}\lessdot x_3\lessdot y_6$ must be increasing. Now suppose that
    $\hat{0}\lessdot x_3\lessdot y_5$ is also increasing. By the
    EL-labeling property, this chain must be lexicographically minimal in
    $[\hat{0},y_5]$, and hence
    \[
    \lambda(\hat{0},x_3)\trianglelefteqslant\lambda(\hat{0},x_2).
    \]
    By Corollary~\ref{cor:1}, these two labels are distinct, so
    \[
    \lambda(\hat{0},x_3)\vartriangleleft\lambda(\hat{0},x_2).
    \]
    It follows that $\hat{0}\lessdot x_3\lessdot y_4$ is also increasing.
    Consequently, $\des_{\hat{0}}(x_3)=0$, again contradicting
    Lemma~\ref{lem:2}.
        
    Therefore, the chains
    \[
    \hat{0}\lessdot x_3\lessdot y_4
    \qquad\text{and}\qquad
    \hat{0}\lessdot x_3\lessdot y_5
    \]
    must be decreasing. Moreover, by Lemma~\ref{lem:1},
    \[
    \lambda(x_3,y_4)\neq\lambda(x_3,y_5).
    \]
    Hence,
    \[
    \operatorname{des}_{\hat{0}}(x_3)=2.
    \]
    On the other hand,
    \[
    \operatorname{des}_{\hat{0}}(x_4)=1.
    \]
    By monotonicity, we must therefore have
    \[
    \lambda(\hat{0},x_4)
    \trianglelefteqslant
    \lambda(\hat{0},x_3).
    \]
    Since these labels are distinct, Lemma~\ref{lem:1} gives
    \[
    \lambda(\hat{0},x_4)
    \vartriangleleft
    \lambda(\hat{0},x_3).
    \]
    
    However, we have already shown that
    \[
    \hat{0}\lessdot x_3\lessdot y_6
    \]
    is increasing. In the interval $[\hat{0},y_6]$, the maximal chain through
    $x_4$ therefore begins with a strictly smaller label than the increasing
    maximal chain through $x_3$. Thus, the increasing maximal chain is not
    lexicographically minimal, contradicting the EL-labeling property.
    
    Hence, $S$ does not admit a monotonic rank-uniform labeling, and therefore
    $S$ is not UMEL-shellable.
\end{proof}
\begin{claim}
    \(S\) is EL-shellable.
\end{claim}
\begin{proof}
    Figure \ref{fig:counter-example-3-EL} below presents an EL-labeling of \(S\) which can be verified by checking the intervals of \(S\).
    \begin{figure}[H]
    \centering
        \resizebox{0.8\textwidth}{!}{
    \begin{tikzpicture}[
        every node/.style={
            circle,
            draw=black,
            fill=red!15,
            inner sep=2pt,
            minimum size=7mm
        },
        edge/.style={
            draw=blue!70!black,
            thick
        },
        edge label/.style={
            draw=none,
            fill=white,
            inner sep=1.2pt,
            font=\small,
            text=blue!70!black
        }
    ]


    \node (one) at (0,5) {$\hat{1}$};

    \node (y1) at (-7,3) {$y_1$};
    \node (y2) at (-5,3) {$y_2$};
    \node (y3) at (-3,3) {$y_3$};
    \node (y4) at (-1,3) {$y_4$};
    \node (y5) at (1,3) {$y_5$};
    \node (y6) at (3,3) {$y_6$};
    \node (y7) at (5,3) {$y_7$};
    \node (y8) at (7,3) {$y_8$};

    \node (x1) at (-5,1) {$x_1$};
    \node (x2) at (-1.7,1) {$x_2$};
    \node (x3) at (1.7,1) {$x_3$};
    \node (x4) at (5,1) {$x_4$};

    \node (zero) at (0,-1.5) {$\hat{0}$};


    \draw[edge] (one) -- (y1)
        node[pos=.58, above left=3pt, edge label] {$1$};

    \draw[edge] (one) -- (y2)
        node[pos=.58, above left=3pt, edge label] {$1$};

    \draw[edge] (one) -- (y3)
        node[pos=.58, above left=2pt, edge label] {$1$};

    \draw[edge] (one) -- (y4)
        node[pos=.58, left=4pt, edge label] {$1$};

    \draw[edge] (one) -- (y5)
        node[pos=.58, right=4pt, edge label] {$1$};

    \draw[edge] (one) -- (y6)
        node[pos=.58, above right=2pt, edge label] {$1$};

    \draw[edge] (one) -- (y7)
        node[pos=.58, above right=3pt, edge label] {$1$};

    \draw[edge] (one) -- (y8)
        node[pos=.58, above right=3pt, edge label] {$1$};


    \draw[edge] (y1) -- (x1)
        node[pos=.58, left=4pt, edge label] {$1$};

    \draw[edge] (y2) -- (x1)
        node[pos=.52, left=0pt, edge label] {$2$};

    \draw[edge] (y3) -- (x1)
        node[pos=.58, above left=0pt, edge label] {$2$};

    \draw[edge] (y3) -- (x2)
        node[pos=.58, left=2pt, edge label] {$1$};

    \draw[edge] (y4) -- (x2)
        node[pos=.52, left=0pt, edge label] {$2$};

    \draw[edge] (y5) -- (x2)
        node[pos=.54, left=2pt, edge label] {$2$};

    \draw[edge] (y4) -- (x3)
        node[pos=.84, left=6pt, edge label] {$1$};

    \draw[edge] (y5) -- (x3)
        node[pos=.52, left=0pt, edge label] {$3$};

    \draw[edge] (y6) -- (x3)
        node[pos=.58, above left=2pt, edge label] {$3$};

    \draw[edge] (y6) -- (x4)
        node[pos=.58, left=4pt, edge label] {$1$};

    \draw[edge] (y7) -- (x4)
        node[pos=.52, left=0pt, edge label] {$4$};

    \draw[edge] (y8) -- (x4)
        node[pos=.50, left=3pt, edge label] {$4$};


    \draw[edge] (zero) -- (x1)
        node[pos=.50, above left=3pt, edge label] {$1$};

    \draw[edge] (zero) -- (x2)
        node[pos=.50, left=4pt, edge label] {$2$};

    \draw[edge] (zero) -- (x3)
        node[pos=.50, right=4pt, edge label] {$3$};

    \draw[edge] (zero) -- (x4)
        node[pos=.50, above right=3pt, edge label] {$4$};
    
    \end{tikzpicture}}
    \caption{EL-labeling of $S$}
    \label{fig:counter-example-3-EL}
\end{figure}
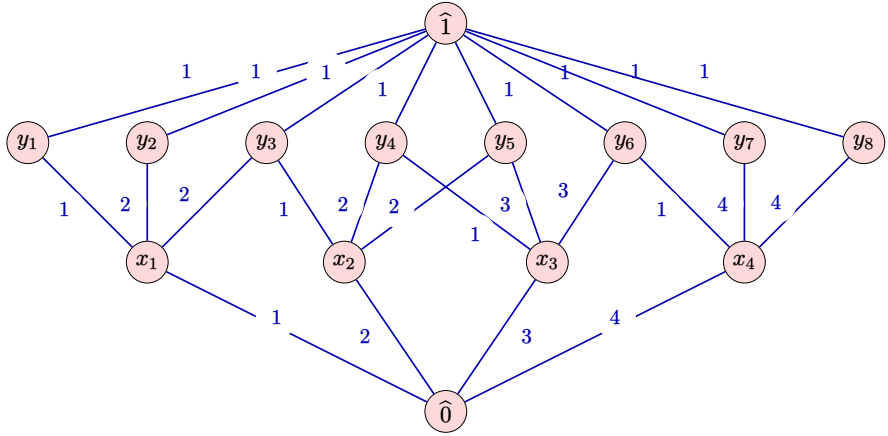
\end{proof}
It is straightforward to verify that every EL-shellable poset of rank \(2\) is UMEL-shellable. The remaining question is whether every rank-\(3\) poset \(P\) admits an EL-labeling \(\lambda\) such that \((P,\lambda)\) is a rank-uniform EL-labeled poset.

\section{Rank-uniform labeled posets of rank-3}
In this section, we prove that every rank-$3$ poset $P$ admits an EL-labeling
$\lambda$ such that $(P,\lambda)$ is a rank-uniform EL-labeled poset. We begin
by introducing the necessary definitions.

\begin{definition}[Atoms Insertion Algorithm]\label{AIA}
    Let $P$ be a bounded graded rank-uniform poset of rank $3$ with \(|\At(P)|=n\) and \( W_1(s)=k\) for every atom $s\in P$. Thus, every atom of $P$ is covered by exactly
    $k$ elements of rank $2$. We may construct $P$ by inserting its atoms one
    at a time as follows.
    
    \begin{itemize}
        \item Choose an ordering of the atoms of $P$,
        \[
        (x_1,x_2,\ldots,x_n).
        \]
    
        \item Begin with the atom $x_1$. Let \(y_1,y_2,\ldots,y_k\) be the $k$ elements covering $x_1$, and define
        \[
        P_1=
        \{\hat{0},x_1,y_1,\ldots,y_k,\hat{1}\},
        \]
        with cover relations
        \[
        \hat{0}\lessdot x_1\lessdot y_j\lessdot\hat{1},
        \qquad 1\leq j\leq k.
        \]
    
        \item Suppose that $P_{i-1}$ has been constructed using the atoms
        $x_1,\ldots,x_{i-1}$. To insert $x_i$, consider the $k$ elements of
        $P$ that cover $x_i$. We partition them into two sets:
        \[
        \operatorname{icov}(x_i)
        =
        \{y_{i_1},\ldots,y_{i_t}\},
        \]
        consisting of those elements that are already present in $P_{i-1}$, and
        \[
        \operatorname{ocov}(x_i)
        =
        \{y_{i_{t+1}},\ldots,y_{i_k}\},
        \]
        consisting of those elements that have not yet appeared and are inserted
        together with $x_i$.
    
        We then obtain $P_i$ from $P_{i-1}$ by adding $x_i$, the elements in
        $\operatorname{ocov}(x_i)$, and the cover relations
        \[
        \hat{0}\lessdot x_i\lessdot y
        \qquad
        \text{for every }
        y\in \operatorname{icov}(x_i)\cup\operatorname{ocov}(x_i),
        \]
        together with
        \[
        y\lessdot\hat{1}
        \qquad
        \text{for every }
        y\in\operatorname{ocov}(x_i).
        \]
    \end{itemize}
    After all $n$ atoms have been inserted, we obtain $P_n=P$.

    \begin{figure}[H]
    \centering
    \resizebox{\textwidth}{!}{
    \begin{tikzpicture}[
        vertex/.style={
            circle,
            draw=black,
            fill=red!15,
            inner sep=2pt,
            minimum size=7mm
        },
        edge/.style={
            draw=blue!70!black,
            thick
        }
    ]

    \begin{scope}[xshift=0cm]

        \node[vertex] (t1) at (0,4) {$\hat{1}$};

        \node[vertex] (y11) at (-1.5,2) {$y_1$};
        \node[vertex] (y12) at (0,2) {$y_2$};
        \node[vertex] (y13) at (1.5,2) {$y_3$};

        \node[vertex] (x11) at (0,0) {$x_1$};

        \node[vertex] (b1) at (0,-2) {$\hat{0}$};

        \draw[edge] (t1) -- (y11);
        \draw[edge] (t1) -- (y12);
        \draw[edge] (t1) -- (y13);

        \draw[edge] (y11) -- (x11);
        \draw[edge] (y12) -- (x11);
        \draw[edge] (y13) -- (x11);

        \draw[edge] (x11) -- (b1);

    \end{scope}


    \draw[->, very thick]
        (2.3,1) -- (3.7,1);

    \begin{scope}[xshift=7cm]

        \node[vertex] (t2) at (0,4) {$\hat{1}$};

        \node[vertex] (y21) at (-2.25,2) {$y_1$};
        \node[vertex] (y22) at (-0.75,2) {$y_2$};
        \node[vertex] (y23) at (0.75,2) {$y_3$};
        \node[vertex] (y24) at (2.25,2) {$y_4$};

        \node[vertex] (x21) at (-1,0) {$x_1$};
        \node[vertex] (x22) at (1,0) {$x_2$};

        \node[vertex] (b2) at (0,-2) {$\hat{0}$};

        \draw[edge] (t2) -- (y21);
        \draw[edge] (t2) -- (y22);
        \draw[edge] (t2) -- (y23);
        \draw[edge] (t2) -- (y24);

        \draw[edge] (y21) -- (x21);
        \draw[edge] (y22) -- (x21);
        \draw[edge] (y23) -- (x21);

        \draw[edge] (y22) -- (x22);
        \draw[edge] (y23) -- (x22);
        \draw[edge] (y24) -- (x22);

        \draw[edge] (x21) -- (b2);
        \draw[edge] (x22) -- (b2);

    \end{scope}


    \draw[->, very thick]
        (10,1) -- (11.4,1);

    \begin{scope}[xshift=15cm]

        \node[vertex] (t3) at (0,4) {$\hat{1}$};

        \node[vertex] (y31) at (-3,2) {$y_1$};
        \node[vertex] (y32) at (-1.5,2) {$y_2$};
        \node[vertex] (y33) at (0,2) {$y_3$};
        \node[vertex] (y34) at (1.5,2) {$y_4$};
        \node[vertex] (y35) at (3,2) {$y_5$};

        \node[vertex] (x31) at (-1.5,0) {$x_1$};
        \node[vertex] (x32) at (0,0) {$x_2$};
        \node[vertex] (x33) at (1.5,0) {$x_3$};

        \node[vertex] (b3) at (0,-2) {$\hat{0}$};

        \draw[edge] (t3) -- (y31);
        \draw[edge] (t3) -- (y32);
        \draw[edge] (t3) -- (y33);
        \draw[edge] (t3) -- (y34);
        \draw[edge] (t3) -- (y35);

        \draw[edge] (y31) -- (x31);
        \draw[edge] (y32) -- (x31);
        \draw[edge] (y33) -- (x31);

        \draw[edge] (y32) -- (x32);
        \draw[edge] (y33) -- (x32);
        \draw[edge] (y34) -- (x32);

        \draw[edge] (y33) -- (x33);
        \draw[edge] (y34) -- (x33);
        \draw[edge] (y35) -- (x33);

        \draw[edge] (x31) -- (b3);
        \draw[edge] (x32) -- (b3);
        \draw[edge] (x33) -- (b3);

    \end{scope}

    \end{tikzpicture}
    }
    \caption{Example of the Atoms Insertion Algorithm}
\end{figure}
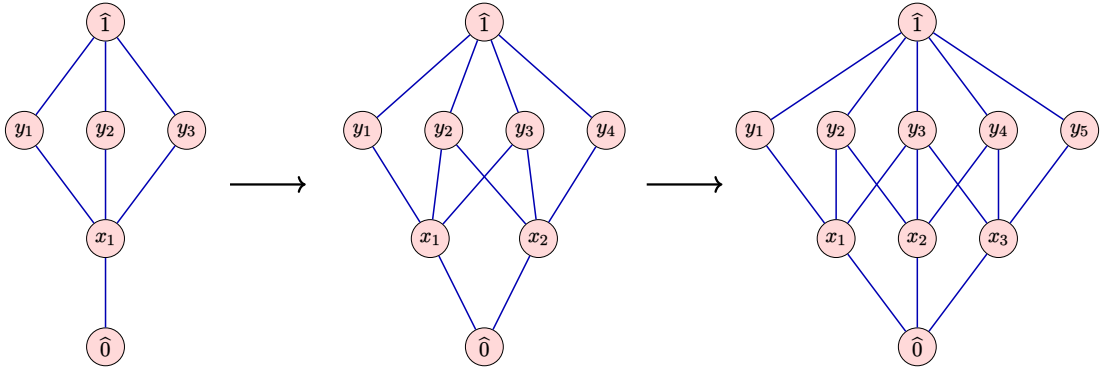
\end{definition}
\begin{claim}
    If \(P\) is EL-shellable, then we can choose an ordering of the atoms \((x_1,x_2,\ldots,x_n)\) such that \[ \operatorname{icov}(x_i)\neq\varnothing \qquad\text{for every } 2\leq i\leq n. \]
\end{claim}
\begin{proof}
    The existence of such an ordering is equivalent to the connectedness of the subgraph of the Hasse diagram induced by the elements of ranks $1$ and $2$. Indeed, since $P$ has rank $3$, this graph is precisely the order complex of
    the proper part of $P$. By a result of Bj\"orner~\cite[Theorem~2.3]{bjorner1980},
    an EL-shellable poset is shellable. Hence, this graph is shellable and,
    being one-dimensional, must be connected. Alternatively, this fact can be
    proved directly using an argument similar to that in the proof of
    Lemma~\ref{lem:2}.
\end{proof}
\begin{maintheorem}\label{theorem-1}
    Let $\P$ be an EL-shellable poset of rank-3. If $\P$ is rank-uniform, then there is an EL-labeling $\lambda$ of $\P$ such that $(\P, \lambda)$ is a rank-uniform EL-labeled poset.
\end{maintheorem}

\begin{proof}
    We construct the labeling $\lambda$ simultaneously with $P$ using the
    Atoms Insertion Algorithm~\ref{AIA}. Let
    \[
    (x_1,x_2,\ldots,x_n)
    \]
    be an ordering of the atoms of $P$ such that
    \[
    \operatorname{icov}(x_i)\neq\varnothing
    \qquad\text{for every }2\leq i\leq n,
    \]
    as guaranteed by the previous claim. Label the edges between ranks $0$ and
    $1$ so that
    \[
    \lambda(\hat{0},x_1)
    \vartriangleleft
    \lambda(\hat{0},x_2)
    \vartriangleleft\cdots\vartriangleleft
    \lambda(\hat{0},x_n).
    \]
    
    We now construct the remaining labels inductively. Let
    $y_1,\ldots,y_k$ be the $k$ elements covering $x_1$. Choose the labels so
    that
    \[
    \lambda(\hat{0},x_1)
    \vartriangleleft
    \lambda(x_1,y_1)
    \vartriangleleft
    \lambda(x_1,y_2)
    \vartriangleleft\cdots\vartriangleleft
    \lambda(x_1,y_k).
    \]
    We then label the edges from rank $2$ to $\hat{1}$ so that
    \[
    \lambda(y_2,\hat{1})
    \vartriangleleft\cdots\vartriangleleft
    \lambda(y_k,\hat{1})
    \vartriangleleft
    \lambda(x_1,y_1)
    \vartriangleleft
    \lambda(y_1,\hat{1}).
    \]
    If $n\geq 2$, we further choose these labels so that
    \[
    \lambda(y_j,\hat{1})
    \vartriangleleft
    \lambda(\hat{0},x_2)
    \qquad\text{for every }1\leq j\leq k.
    \]
    Thus, $x_1\lessdot y_1\lessdot\hat{1}$ is the unique increasing maximal
    chain in $[x_1,\hat{1}]$, while
    \[
    \hat{0}\lessdot x_1\lessdot y_1\lessdot\hat{1}
    \]
    is the unique increasing maximal chain in $[\hat{0},\hat{1}]$.
    
    Suppose now that the edges of $P_{i-1}$ have been labeled, and write
    \[
    \operatorname{icov}(x_i)
    =
    \{y_{i_1},\ldots,y_{i_t}\},
    \qquad
    \operatorname{ocov}(x_i)
    =
    \{y_{i_{t+1}},\ldots,y_{i_k}\}.
    \]
    By the induction hypothesis,
    \[
    \lambda(y,\hat{1})
    \vartriangleleft
    \lambda(\hat{0},x_i)
    \qquad
    \text{for every }y\in\operatorname{icov}(x_i).
    \]
    Since $\operatorname{icov}(x_i)\neq\varnothing$, choose
    $y_{i_1}\in\operatorname{icov}(x_i)$ and assign the labels on the edges
    from $x_i$ to $\operatorname{icov}(x_i)$ so that
    \[
    \lambda(x_i,y_{i_1})
    \vartriangleleft
    \lambda(y_{i_1},\hat{1})
    \vartriangleleft
    \lambda(\hat{0},x_i),
    \]
    whereas, for $2\leq r\leq t$,
    \[
    \lambda(y_{i_r},\hat{1})
    \vartriangleleft
    \lambda(x_i,y_{i_r})
    \vartriangleleft
    \lambda(\hat{0},x_i).
    \]
    We may choose these labels so that
    \[
    \lambda(x_i,y_{i_1})
    \vartriangleleft
    \lambda(x_i,y_{i_2})
    \vartriangleleft\cdots\vartriangleleft
    \lambda(x_i,y_{i_t}).
    \]
    Hence,
    \[
    x_i\lessdot y_{i_1}\lessdot\hat{1}
    \]
    is the unique increasing maximal chain in $[x_i,\hat{1}]$, and it is
    lexicographically minimal.
    
    For the elements in $\operatorname{ocov}(x_i)$, choose the labels so that
    \[
    \lambda(\hat{0},x_i)
    \vartriangleleft
    \lambda(x_i,y_{i_{t+1}})
    \vartriangleleft\cdots\vartriangleleft
    \lambda(x_i,y_{i_k}),
    \]
    and assign their upper labels so that
    \[
    \lambda(y_{i_{t+1}},\hat{1})
    \vartriangleleft\cdots\vartriangleleft
    \lambda(y_{i_k},\hat{1})
    \vartriangleleft
    \lambda(\hat{0},x_i).
    \]
    Consequently, every chain through an element of
    $\operatorname{ocov}(x_i)$ has a descent before reaching $\hat{1}$.
    Moreover, if $i<n$, then
    \[
    \lambda(y,\hat{1})
    \vartriangleleft
    \lambda(\hat{0},x_i)
    \vartriangleleft
    \lambda(\hat{0},x_{i+1})
    \]
    for every $y\in\operatorname{ocov}(x_i)$, so the induction hypothesis is
    preserved.
    
    It remains to verify that $\lambda$ is an EL-labeling. For an interval
    $[\hat{0},y]$ with $y$ of rank $2$, let $x_i$ be the first atom in the
    chosen ordering that is covered by $y$. At the stage when $y$ is first
    inserted, we have
    \[
    \lambda(\hat{0},x_i)
    \vartriangleleft
    \lambda(x_i,y),
    \]
    so $\hat{0}\lessdot x_i\lessdot y$ is increasing. For every later atom
    $x_j$ covered by $y$, the element $y$ belongs to
    $\operatorname{icov}(x_j)$, and therefore
    \[
    \lambda(x_j,y)
    \vartriangleleft
    \lambda(\hat{0},x_j).
    \]
    Thus, $\hat{0}\lessdot x_i\lessdot y$ is the unique increasing maximal
    chain in $[\hat{0},y]$. Since $x_i$ is the first atom covering $y$, this
    chain is also lexicographically minimal.
    
    Similarly, for every atom $x_i$, the construction makes
    \[
    x_i\lessdot y_{i_1}\lessdot\hat{1}
    \]
    the unique increasing and lexicographically minimal maximal chain in
    $[x_i,\hat{1}]$. Finally,
    \[
    \hat{0}\lessdot x_1\lessdot y_1\lessdot\hat{1}
    \]
    is the unique increasing maximal chain in $[\hat{0},\hat{1}]$, and it is
    lexicographically minimal because $\lambda(\hat{0},x_1)$ is the smallest
    label on an edge covering $\hat{0}$. Hence, $\lambda$ is an EL-labeling of
    $P$.
    
    It remains to show that $(P,\lambda)$ is rank-uniform as a labeled poset.
    By construction, the labels on the edges above each atom are distinct.
    Since every atom is covered by exactly $k$ elements, the width condition is
    therefore the same for every atom. Furthermore, for each atom $x_i$, there
    is exactly one element $y$ covering $x_i$ for which
    \[
    \operatorname{des}_{x_i}(y)=0,
    \]
    namely the element defining the unique increasing chain in
    $[x_i,\hat{1}]$, while
    \[
    \operatorname{des}_{x_i}(y)=1
    \]
    for every other element $y$ covering $x_i$. Thus, the descent condition is
    also uniform over all atoms. Therefore, $(P,\lambda)$ is a rank-uniform
    EL-labeled poset.
\end{proof}

Thus, for every bounded graded rank-uniform poset $P$ of rank $3$, we have
shown that it is sufficient to find an ordering of the atoms
\[
(x_1,x_2,\ldots,x_n)
\]
such that
\[
\operatorname{icov}(x_i)\neq\varnothing
\qquad\text{for every } 2\leq i\leq n,
\]
in order to construct an EL-labeling $\lambda$ for which $(P,\lambda)$ is
a rank-uniform EL-labeled poset.

Moreover, if the ordering satisfies the additional condition that
$|\operatorname{icov}(x_i)|$ is weakly increasing with $i$, then the same
construction yields a monotonic rank-uniform EL-labeling. Such an ordering
does not always exist, as demonstrated by the counterexample in
Figure~\ref{fig:counter-example-3}.

\begin{maintheorem}\label{theorem-2}
    Let $P$ be a bounded graded rank-uniform poset of rank $3$. If its atoms can be ordered as $(x_1,x_2,\ldots,x_n)$ so that
    \[ 0<|\operatorname{icov}(x_2)| \leq|\operatorname{icov}(x_3)| \leq\cdots\leq|\operatorname{icov}(x_n)|, \]
    Then $P$ is UMEL-shellable.
\end{maintheorem}

The preceding construction is closely related to Question~8.3 of
Coron, Ferroni, and Li~\cite{coron2025chow}. In rank $3$, it yields a
result that is stronger than the statement proposed in
Question~8.4, since it applies to a broader class of posets than
geometric lattices.

\begin{maintheorem}\label{theorem-3}
Let $P$ be a bounded graded rank-uniform poset of rank $3$. Suppose that,
for every ordering $(x_1,x_2,\ldots,x_n)$ of the atoms of $P$,
\[
\operatorname{icov}(x_i)\neq\varnothing
\qquad\text{for every }2\leq i\leq n.
\]
Then $P$ is UMEL-shellable.
\end{maintheorem}

\begin{proof}
We construct an ordering $(x_1,x_2,\ldots,x_n)$ of the atoms greedily.
Choose $x_1$ arbitrarily. Suppose that
$x_1,x_2,\ldots,x_{i-1}$ have already been chosen. Among the remaining
atoms, choose $x_i$ so that
\[
|\operatorname{icov}(x_i)|
\]
is minimal.

By hypothesis,
\[
|\operatorname{icov}(x_i)|>0
\qquad\text{for every }i\geq 2.
\]
Moreover, as more atoms are inserted, the quantity
$|\operatorname{icov}(s)|$ can only increase for each remaining atom $s$.
Since $x_i$ is chosen to minimize this quantity among all remaining atoms,
it follows that
\[
0<
|\operatorname{icov}(x_2)|
\leq
|\operatorname{icov}(x_3)|
\leq\cdots\leq
|\operatorname{icov}(x_n)|.
\]
Therefore, Theorem~\ref{theorem-2} implies that $P$ is UMEL-shellable.
\end{proof}

For a geometric lattice $\mathcal{L}$ of rank $3$, every pair of distinct
atoms $s_1,s_2$ has a common rank-$2$ upper cover, namely their join
$s_1\vee s_2$. Consequently, for every ordering
$(x_1,x_2,\ldots,x_n)$ of the atoms of $\mathcal{L}$,
\[
\operatorname{icov}(x_i)\neq\varnothing
\qquad\text{for every }2\leq i\leq n.
\]
Hence, Theorem~\ref{theorem-3} gives an affirmative answer to
Question~8.4 of Coron-Ferroni-Li~\cite{coron2025chow}.

\subsection*{AI Disclosure} 
Claude AI was used to assist in developing and running a computer program for searching for counterexamples. ChatGPT was used to assist with the generation of TikZ code for the figures, as well as with rephrasing and improving the formal presentation of the manuscript.
\bibliographystyle{amsalpha}
\bibliography{biblio}

\end{document}